\documentclass[preprint,12pt]{elsarticle}

\usepackage{amssymb}
\usepackage{amssymb, amsmath, mathrsfs}
\usepackage{amscd}
\usepackage{latexsym}
\usepackage{amssymb}
\usepackage{amsmath}
\usepackage{amsthm}
\usepackage{indentfirst}
\usepackage{color}

\allowdisplaybreaks

\newtheorem{theorem}{\indent Theorem}[section]
\newtheorem{lemma}{\indent Lemma}[section]\rm
\rm
\rm
\newtheorem{remark}{\indent Remark}[section]\rm
\rm
\rm

\journal{}

\begin{document}

\begin{frontmatter}



\title{Solitary wave of the Schr\"{o}dinger lattice system with nonlinear
hopping}


\author{Ming Cheng}\cortext[cor1]{Corresponding author} \ead{jlumcheng@hotmail.com}
\address{College of Mathematics, Jilin University, Changchun 130012, P.R. China}

\begin{abstract}
This paper is concerned with the nonlinear Schr\"{o}dinger lattice with nonlinear hopping. Via variation approach and the Nehari manifold argument, we obtain two types of solution: periodic ground state and localized ground state. Moreover, we consider the convergence of periodic solutions to the solitary wave.
\end{abstract}

\end{frontmatter}


\section{Introduction}
In the last decades, a great deal of attention has been paid to study the existence of solitary wave for the lattice systems\cite{Joh,Aub,Zha}. They play a role in lots of physical models, such as nonlinear waves in crystals and arrays of coupled optical
waveguides. The discrete nonlinear Schr\"{o}dinger lattice is one of the most famous models in
mathematics and physics. The existence and properties of  discrete
breathers(periodic in time and spatially localized) in discrete nonlinear Schr\"{o}dinger lattice have been
considered in a number of studies\cite{Cue,Kara2}.

In the present paper, we consider a variant of the discrete nonlinear Schr\"{o}dinger lattice as follows:
\begin{align}\label{1.1}
i\dot{\psi}_{l}+(\triangle_{d}\psi)_{l}+\alpha\psi_{l}\sum_{j=1}^{d}(\mathcal{T}_{j}\psi)_{l}+\beta|\psi_{l}|^{2\sigma}\psi_{l}=0,\quad l\in\mathbb{Z}^{d},
\end{align}
where $\alpha,\beta\in\mathbb{R}$, $(\triangle_{d}\psi)_{l}=\sum_{m\in N_{d}}\psi_{m}-\psi_{l}$ and  the nonlinear operator $\mathcal{T}$ is defined by
$$
(\mathcal{T}_{j}\psi)_{l}=|\psi_{l_{1},\ldots,l_{j-1},l_{j}+1,l_{j+1},\ldots,l_{d}}|^{2}+|\psi_{l_{1},\ldots,l_{j-1},l_{j}-1,l_{j+1},\ldots,l_{d}}|^{2}.
$$
Here, $N_{d}$ denotes the set of the nearest neighbors of the point $l\in\mathbb{Z}^{d}$.

Note that for $\alpha=0,\beta\neq 0$, it recovers the classical  nonlinear Schr\"{o}dinger lattice. For $\alpha\neq 0,\beta\neq 0$, it denotes the  Schr\"{o}dinger lattice with nonlinear hopping.

There has been a lot of  interests in this equation as the modeling of  waveguide arrays.  Also, nonlinear hopping terms appear from Klein-Gordon and Fermi-Pasta-Ulam chains of anharmonic oscillators coupled with anharmonic inter-site potentials, or mixed FPU/KG chains.  N. I. Karachalios \emph{et al.}  discuss the energy thresholds in the setting of DNLS lattice with  nonlinear  hopping terms by using fixed point method. The numerical results have also been obtained in their paper\cite{Kara}.

Our aim is to investigate the existence of nontrivial solitary wave for the  infinite dimensional lattice  \eqref{1.1}. Here, we only consider the case of one dimension. i. e., $d=1$. The  case of $d>1$ is similar. It notes that for classical nonlinear Schr\"{o}dinger lattice, Weinstein\cite{Wein} discusses a connection among the dimensionality, the degree of the nonlinearity and the existence of the excitation threshold. They prove that if the degree of the nonlinearity $\sigma$ satisfies $\sigma\geq \frac{2}{d}$ where $d$ is the dimension, then there exists a ground state for the total power is greater than the excitation threshold and there is no ground state for the total power is less than the excitation threshold. However, we get that the power of solitary wave always has a lower bound for the equation \eqref{1.1} with $\sigma\geq 1$.

The paper is organized as follows. In Section 2, we firstly consider the $k$-periodic problem. Note that the dimension in space variable is finite. We obtain the nontrival periodic solution by Nehari manifolds argument\cite{Neh}. The existence of solitary wave is more complex. In Section 3, we follow the idea of \cite{Pan1,Pan2,Pan3,Pan4} to obtain the solitary wave. The key point is to show the norms of periodic ground state are bounded. It is based on the concentration compactness. In  Section 4, we concern the convergence of periodic ground states to a solitary ground state.

\section{Periodic solution}

In this paper, we consider the the following equation:
\begin{align}\label{2.2}
i\dot{\psi}_{l}+(\triangle_{d}\psi)_{l}+\alpha \psi_{l}(|\psi_{l+1}|^{2}+|\psi_{l-1}|^{2})+\beta|\psi_{l}|^{2\sigma}\psi_{l}=0, \quad l\in\mathbb{Z},
\end{align}
where $\sigma\geq 1$.

To obtain breather, we seek the solution:
$$
\psi_{l}=e^{-i\omega t}u_{l}.
$$
The equation of $u_{l}$ is
\begin{align}\label{2.3}
\omega u_{l}+(\triangle_{d}u)_{l}+\alpha u_{l}(|u_{l+1}|^{2}+|u_{l-1}|^{2})+\beta|u_{l}|^{2\sigma}u_{l}=0, \quad l\in\mathbb{Z}.
\end{align}
Actually, we give the proofs only in the focusing  case with  $\alpha,\beta>0$ and $\omega<0$. For the defocusing case with $\alpha,\beta<0$ and $\omega>4$, the argument is similar. Here, we omit the details.

In this section, we prove the existence of $k$-periodic solution which satisfies
$$
u_{l+k}=u_{l},\quad\mbox{for}\quad l\in\mathbb{Z},
$$
where $k>2$ is an integer.

Let
$$
P_{k}=\Big\{l\in\mathbb{Z}|-[\frac{k}{2}]\leq l\leq k-[\frac{k}{2}]-1\Big\}.
$$
Consider the Banach space $l_{k}^{p}$ with norm:
$$
||u||_{l_{k}^{p}}^{p}=\sum_{l\in P_{k}}|u_{l}|^{p}.
$$
We mention that
$$
||u||_{l_{k}^{q}}\leq ||u||_{l_{k}^{p}},\quad 1\leq p\leq q\leq \infty.
$$
Denote that $\langle\cdot,\cdot\rangle_{k}$ is natural inner product in $l^{2}_{k}$.

Define the  functional
\begin{align*}
J_{k}(u)=\langle-\triangle_{d}u,u\rangle_{k}-\omega\langle u,u\rangle_{k}-\alpha\sum_{l\in P_{k}}|u_{l}|^{2}|u_{l+1}|^{2}-\frac{\beta}{\sigma+1}\sum_{l\in P_{k}}|u_{l}|^{2\sigma+2},
\end{align*}
and Nehari manifold
\begin{align*}
\mathcal{N}_{k}=\Big\{&u\in l^{2}_{k}|I_{k}(u)=\langle-\triangle_{d}u,u\rangle_{k}-\omega\langle u,u\rangle_{k}-\alpha\sum_{l\in P_{k}}|u_{l}|^{2}(|u_{l+1}|^{2}+|u_{l-1}|^{2})\nonumber\\
&-\beta\sum_{l\in P_{k}}|u_{l}|^{2\sigma+2}=0,u\neq 0\Big\}.
\end{align*}
Then, the minimizer of the constrained variational problem:
\begin{align*}
m_{k}=\inf_{u\in \mathcal{N}_{k}}\{J_{k}(u)\}
\end{align*}
is the nontrivial periodic solution of $\eqref{2.3}$. We mention that the minimizer  is called a periodic ground state.

Note that
$$
0 \leq \langle\triangle_{d}u,u\rangle_{k} \leq 4||u||^{2}_{l^{2}_{k}},\quad \mbox{for}\quad u\in l^{2}_{k}.
$$
We want to obtain the periodic solution with prescribed frequency $\omega<0$.
With the Nehari manifold approach, we have the following result.
\begin{theorem}\label{th2.1}
Assume that the frequency $\omega<0$ and $\alpha,\beta>0$.
There exists a positive $k$-periodic ground state $u^{k}$ for the equation $\eqref{2.3}$.
\end{theorem}

\begin{lemma}\label{lem2.1}
Under the assumptions of Theorem \ref{th2.1}. The Nehari manifold $\mathcal{N}_{k}$ is nonempty.
\end{lemma}
\begin{proof}
For $t\geq 0$ and $u\neq 0$, define
\begin{align*}
\rho(t)&=I_{k}(\sqrt{t}u)\\
&=t(\langle-\triangle_{d}u,u\rangle_{k}-\omega\langle u,u\rangle_{k})-2t^{2}\alpha\sum_{l\in P_{k}}|u_{l}|^{2}|u_{l+1}|^{2}-t^{\sigma+1}\beta\sum_{l\in P_{k}}|u_{l}|^{2\sigma+2}.
\end{align*}
Then,
\begin{align*}
\rho'(t)=\langle-\triangle_{d}u,u\rangle_{k}-\omega\langle u,u\rangle_{k}-4t\alpha\sum_{l\in P_{k}}|u_{l}|^{2}|u_{l+1}|^{2}-(\sigma+1)t^{\sigma}\beta\sum_{l\in P_{k}}|u_{l}|^{2\sigma+2}.
\end{align*}
There holds that $\rho'(t)>0$ for $t>0$ small enough.

Observe that
\begin{align*}
\rho''(t)=-4\alpha\sum_{l\in P_{k}}|u_{l}|^{2}|u_{l+1}|^{2}-(\sigma+1)\sigma t^{\sigma-1}\beta\sum_{l\in P_{k}}|u_{l}|^{2\sigma+2}<0.
\end{align*}
Therefore, $\rho(t)$ admits a
unique zero point  $t^{*}\in (0,+\infty)$. This implies $\sqrt{t^{*}}u\in \mathcal{N}_{k}$. It completes the proof.
\end{proof}

\begin{lemma}\label{lem2.2}
Under the assumptions of Theorem \ref{th2.1}. For $u\in\mathcal{N}_{k}$, the function $J_{k}(\sqrt{t}u)$ has a unique critical point at $t=1$, which is a global maximum.
\end{lemma}
\begin{proof}
For $t>0$ and $u\in\mathcal{N}_{k}$, we get
\begin{align*}
\theta(t)&=J_{k}(\sqrt{t}u)\\
&=t(\langle-\triangle_{d}u,u\rangle_{k}-\omega\langle u,u\rangle_{k})-t^{2}\alpha\sum_{l\in P_{k}}|u_{l}|^{2}|u_{l+1}|^{2}-\frac{t^{\sigma+1}\beta}{\sigma+1}\sum_{l\in P_{k}}|u_{l}|^{2\sigma+2}.
\end{align*}
Then,
\begin{align*}
\theta'(t)=\langle-\triangle_{d}u,u\rangle_{k}-\omega\langle u,u\rangle_{k}-2t\alpha\sum_{l\in P_{k}}|u_{l}|^{2}|u_{l+1}|^{2}-t^{\sigma}\beta\sum_{l\in P_{k}}|u_{l}|^{2\sigma+2}.
\end{align*}
It holds that $\theta'(t)>0$ for $t>0$ small enough.

Note that
\begin{align*}
\theta''(t)=-2\alpha\sum_{l\in P_{k}}|u_{l}|^{2}|u_{l+1}|^{2}-\sigma t^{\sigma-1}\beta\sum_{l\in P_{k}}|u_{l}|^{2\sigma+2}<0.
\end{align*}
We can see that $t=1$ is the unique maximum point of $\theta(t)$. This implies the proof.
\end{proof}

Assume that $u^{k}$ is the $k$-periodic solution of \eqref{2.3}, we have
\begin{align}\label{2.6}
&|\omega|||u^{k}||^{2}_{l^{2}_{k}}\nonumber\\
\leq &\langle-\triangle_{d}u^{k},u^{k}\rangle_{k}-\omega\langle u^{k},u^{k}\rangle_{k}\nonumber\\
=&2\alpha\sum_{l\in P_{k}}|u^{k}_{l}|^{2}|u^{k}_{l+1}|^{2}+\beta\sum_{l\in P_{k}}|u^{k}_{l}|^{2\sigma+2}\nonumber\\
\leq&||u^{k}||^{2}_{l^{2}_{k}}\Big(\beta||u^{k}||_{l^{2}_{k}}^{2\sigma}+2\alpha||u^{k}||^{2}_{l^{2}_{k}}\Big).
\end{align}
Therefore
\begin{align}\label{2.7}
||u^{k}||_{l^{2}_{k}}\geq C_{1}>0.
\end{align}
where $C_{1}$ is the unique positive solution of equation:
$$
\beta x^{2\sigma}+2\alpha x^{2}+\omega=0.
$$
Observe that $C_{1}$ is independent of $k$.

Thus, we get a lower bound of the power of the periodic solutions.
\begin{theorem}
The power of the periodic solution must be greater than $C_{1}$.
\end{theorem}
\begin{lemma}
Under the assumptions of Theorem \ref{th2.1}.
$J_{k}(u)$ is bounded below for all $u\in \mathcal{N}_{k}$.
\end{lemma}
\begin{proof}
Let $u\in \mathcal{N}_{k}$. From the argument in \eqref{2.6} and \eqref{2.7}, there exists $l_{0}\in P_{k}$ and a positive constant $C_{2}$ such that
$$
|u_{l_{0}}|>C_{2}>0.
$$
Therefore,
\begin{align*}
J_{k}(u)=\alpha\sum_{P_{k}}|u_{l}|^{2}|u_{l-1}|^{2}+\frac{\sigma\beta}{\sigma+1}\sum_{P_{k}}|u_{l}|^{2\sigma+2}>\frac{\sigma\beta}{\sigma+1}C_{2}^{2\sigma+2}.
\end{align*}
It completes the proof.
\end{proof}

\begin{lemma}
Under the assumptions of Theorem \ref{th2.1}. Then, the minimizer of the constrained variational problem
$m_{k}=\inf_{u\in \mathcal{N}_{k}}\{J_{k}(u)\}$ could be attained.
\end{lemma}
\begin{proof}
Assume that $\{u^{n}\}$ is a minimizing sequence.
We can see that there exists a  constant $M>0$ such that
$$
\max J_{k}(u^{n})\leq M.
$$
Thus,
$$
|\omega|||u^{n}||_{l^{2}_{k}}\leq \langle-\triangle u^{n},u^{n}\rangle_{k}-\omega\langle u^{n},u^{n}\rangle_{k}\leq M.
$$
There holds that $||u^{n}||_{l^{\infty}_{k}}$ is bounded.

Note that $P_{k}$ is finite dimensional space. Passing to a subsequence, there exists $u^{k}$ such that $u^{n_{j}}\rightarrow u^{k}$ in $l^{2}_{k}$.
Since the set $l^{2}_{k}$ is closed and the functional $J_{k}$ is continuous,
we obtain that $u^{k}\in\mathcal{N}_{k}$ and $J_{k}(u^{k})=m_{k}$.
\end{proof}

By Lagrange multiplier method,
there exists some constant $\lambda$ such that
\begin{align*}
&\lambda\Bigg(2\langle-\triangle_{d}u^{k},v\rangle_{k}-2\omega\langle u^{k},v\rangle_{k}-4\alpha\sum_{l\in P_{k}}u^{k}_{l}v_{l}(|u^{k}_{l+1}|^{2}+|u^{k}_{l-1}|^{2})
-(2\sigma+2)\beta\sum_{l\in P_{k}}|u^{k}_{l}|^{2\sigma}u^{k}_{l}v_{l}\Bigg)\\
&+2\langle-\triangle_{d}u^{k},v\rangle_{k}-2\omega\langle u^{k},v\rangle_{k}-2\alpha\sum_{l\in P_{k}}u^{k}_{l}v_{l}(|u^{k}_{l+1}|^{2}+|u^{k}_{l-1}|^{2})-2\beta\sum_{l\in P_{k}}|u^{k}_{l}|^{2\sigma}u^{k}_{l}v_{l}
=0.
\end{align*}
Choose $v=u^{k}$. Note that $u^{k}\in \mathcal{N}_{k}$, there holds
\begin{align*}
\lambda\Bigg(-2\alpha\sum_{l\in P_{k}}|u^{k}_{l}|^{2}(|u^{k}_{l+1}|^{2}+|u^{k}_{l-1}|^{2})
-\sigma\beta\sum_{l\in P_{k}}|u^{k}_{l}|^{2\sigma+2}\Bigg)=0.
\end{align*}
We have $\lambda=0$. It implies that $u^{k}$ is a nontrival solution of equation \eqref{2.2}.

Now, we prove that $u^{k}$ is positive. Observe that
$$
\langle-\triangle_{d}|u|,|u|\rangle-\omega\langle |u|,|u|\rangle\leq \langle-\triangle_{d}u,u\rangle-\omega\langle u,u\rangle.
$$

Since that $u^{k}$ is the nontrival solution. Then, there exists $t^{**}\in(0,1]$ such that $\sqrt{t^{**}}|u^{k}|\in\mathcal{N}_{k}$. It is obvious that

$$
J_{k}(\sqrt{t^{**}}|u^{k}|)\leq m_{k}.
$$
Hence $J_{k}(\sqrt{t^{**}}|u^{k}|)= m_{k}.$ We can assume that $u^{k}=\sqrt{t^{**}}|u^{k}|$.

Let $G(n,m)$ be the Green function of $-\triangle_{d}-\omega$. From \cite{Tes}, we have $G(n,m)>0$ for $\omega<0$.
It obtains that
$$
u^{k}_{n}=\sum_{l\in\mathbb{Z}}G(n,l)\Big(\alpha u^{k}_{l}(|u^{k}_{l+1}|^{2}+|u^{k}_{l-1}|^{2})+\beta|u^{k}_{l}|^{2\sigma}u^{k}_{l}\Big),\quad n\in\mathbb{Z}.
$$
Since that $u^{k}$ is nonnegative, there holds  $u^{k}_{n}>0$ for all $n\in\mathbb{Z}$. It completes the proof of Theorem \ref{th2.1}.

\section{Localized ground state}

Here, we give some notations. Define the  functional
\begin{align*}
J(u)=\langle-\triangle_{d}u,u\rangle-\omega\langle u,u\rangle-\alpha\sum_{l\in \mathbb{Z}}|u_{l}|^{2}|u_{l+1}|^{2}-\frac{\beta}{\sigma+1}\sum_{l\in \mathbb{Z}}|u_{l}|^{2\sigma+2},
\end{align*}
and Nehari manifold
\begin{align*}
\mathcal{N}=\Big\{&u\in l^{2}|I(u)=\langle-\triangle_{d}u,u\rangle-\omega\langle u,u\rangle-\alpha\sum_{l\in \mathbb{Z}}|u_{l}|^{2}(|u_{l+1}|^{2}+|u_{l-1}|^{2})\nonumber\\
&-\beta\sum_{l\in \mathbb{Z}}|u_{l}|^{2\sigma+2}=0,u\neq 0\Big\},
\end{align*}
where $\langle\cdot,\cdot\rangle$ is natural inner product in $l^{2}$.
Thus, we can see that the minimizer of the constrained variational problem:
\begin{align*}
m=\inf_{u\in \mathcal{N}}\{J(u)\}
\end{align*}
is the nontrivial solitary wave of $\eqref{2.3}$. We call this minimizer   a localized ground state.  Similar with Lemma \ref{lem2.1} and \ref{lem2.2}, the results are obtained by replacing $J_{k}(u)$, $I_{k}(u)$ to $J(u)$, $I(u)$.

In this section, to obtain the localized ground state $u$ satisfying
$$
\lim_{l\rightarrow\infty}|u_{l}|=0,
$$
we follow the idea of \cite{Pan1}. We want to pass to the limit as $k\rightarrow\infty$. The key point  is  the following result.

\begin{lemma}\label{lem3.1}
Under the assumptions of Theorem \ref{th2.1}. Let $u^{k}$ be the $k$-periodic solution. Therefore, the sequences $m_{k}$ and $||u^{k}||_{l^{2}_{k}}$ are bounded.
\end{lemma}
\begin{proof}
First, we concern the sequences $m_{k}$ are bounded. From the similar argument of Lemma \ref{lem2.1}, there holds that
for any given $u\in l^{2}$, there exists $t'$ such $I(\sqrt{t'}u)<0$. Since the sequences with finite support are dense in $l^{2}$. Therefore, there exists $\tilde{u}$ with finite support such that
$I(\tilde{u})<0$. It obtains that there exists $t''$ such that $I(\sqrt{t''}\tilde{u})=0$. For $k$ large enough, we have $\mathrm{supp}\sqrt{t''}\tilde{u}\subset P_{k}$. We can get
$\tilde{v}^{k}\in l^{2}_{k}$ such that $\tilde{v}^{k}_{l}=\sqrt{t''}\tilde{u}_{l}$ for $l\in P_{k}$. There holds that $I_{k}(\tilde{v}^{k})=I(\sqrt{t''}\tilde{u}_{l})=0$. And $m_{k}\leq J_{k}(\tilde{v}^{k})= J(\sqrt{t''}\tilde{u})$ is bounded.

Second, we prove that $||u^{k}||_{l^{2}_{k}}$ is uniformly bounded. Assume that $||u^{k}||_{l^{2}_{k}}$ is  unbounded. Passing to a subsequence which is still denoted by itself, we have $||u^{k}||_{l^{2}_{k}}\rightarrow\infty$ for $k\rightarrow\infty$. Let $v^{k}=\frac{u^{k}}{||u^{k}||_{l^{2}_{k}}}$. One of the following should holds:

(i) $v^{k}$ is vanishing, i.e. $||v^{k}||_{l^{\infty}}\rightarrow 0$.

(ii) $v^{k}$ is not vanishing. Passing to a subsequence which is still denoted by itself, there exists $\delta>0$ and $b^{k}\in\mathbb{Z}$ such that $|v_{b^{k}}^{k}|>\delta$ for all $k$.

Now, we rule out the case (i). There holds that
\begin{align*}
0=\frac{I_{k}(u^{k})}{||u^{k}||_{l^{k}}^{2}}=\langle(-\triangle_{d}-\omega) v^{k},v^{k}\rangle_{k}-2\alpha\sum_{l\in P^{k}}|v^{k}_{l}|^{2}|u^{k}_{l-1}|^{2}-\beta\sum_{l\in P_{k}}|v^{k}_{l}|^{2}|u^{k}_{l}|^{2\sigma}.
\end{align*}
Hence,
\begin{align}\label{3.6}
|\omega|=|\omega|||v^{k}||_{l^{2}_{k}}\leq \langle(-\triangle_{d}-\omega) v^{k},v^{k}\rangle_{k}= 2\alpha\sum_{l\in P_{k}}|v^{k}_{l}|^{2}|u^{k}_{l-1}|^{2}+\beta\sum_{l\in P_{k}}|v^{k}_{l}|^{2}|u^{k}_{l}|^{2\sigma}.
\end{align}
Assume that
\begin{align*}
A_{k}&=\Big\{l\in P_{k}\Big||u_{l}^{k}|<M_{0}\Big\},\\
B_{k}&=P_{k}\setminus A_{k},
\end{align*}
where $M_{0}>0$ is a constant which is defined below.

Let $M_{0}$ be small enough such that
\begin{align*}
&2\alpha\sum_{l\in A_{k}}|v^{k}_{l}|^{2}|u^{k}_{l-1}|^{2}+\beta\sum_{l\in A_{k}}|v^{k}_{l}|^{2}|u^{k}_{l}|^{2\sigma}\\
< &2\alpha M_{0}^{2}\sum_{l\in A_{k}}|v^{k}_{l+1}|^{2}+\beta M_{0}^{2\sigma}\sum_{l\in A_{k}}|v^{k}_{l}|^{2}\\
\leq &\frac{|\omega|}{2}.
\end{align*}
Combine with the equation \eqref{3.6}, we have
\begin{align}\label{3.7}
\frac{|\omega|}{2}&\leq\liminf_{k\rightarrow\infty}2\alpha\sum_{l\in B_{k}}|v^{k}_{l+1}|^{2}|u^{k}_{l}|^{2}+\beta\sum_{l\in B_{k}}|v^{k}_{l}|^{2}|u^{k}_{l}|^{2\sigma}.
\end{align}
From the arguement above, there exists a constant $M'>0$ such that
\begin{align*}
m_{k}=\sum_{l\in P_{k}}\alpha|u^{k}_{l}|^{2}|u^{k}_{l+1}|^{2}+\frac{\sigma\beta}{\sigma+1}\sum_{l\in P_{k}}|u^{k}_{l}|^{2\sigma+2}<M'.
\end{align*}
Hence, $||u^{k}||_{l^{2\sigma+2}_{k}}$ is uniformly bounded.

By H\"{o}lder's inequality, we have
\begin{align*}
\sum_{n} |v_{n}|^{2}|u_{n}|^{2}\leq ||u||^{2}_{l^{2\sigma+2}}||v||^{2}_{l^{\frac{2\sigma+2}{\sigma}}},
\end{align*}
\begin{align*}
\sum_{n} |v_{n}|^{2}|u_{n}|^{2\sigma}\leq ||u||^{2\sigma}_{l^{2\sigma+2}}||v||^{2}_{l^{2\sigma+2}}
\end{align*}
and
$$
||v||_{l^{p}}\leq ||v||_{l^{\infty}}^{\frac{p-2}{p}}||v||_{l^{2}}^{\frac{2}{p}},\quad \mbox{for}~~~p>2.
$$
Since $v^{k}$ is vanishing, we can see that
$$
\lim_{k\rightarrow \infty}||v^{k}||_{l^{p}_{k}}=0,\quad \mbox{for}~~~p>2.
$$
It concludes that
$$
\liminf_{k\rightarrow\infty}2\alpha\sum_{B_{k}}|v^{k}_{l+1}|^{2}|u^{k}_{l}|^{2}+\beta\sum_{B_{k}}|v^{k}_{l}|^{2}|u^{k}_{l}|^{2\sigma}\rightarrow 0,\quad\mbox{as}~~k\rightarrow\infty.
$$
It contradicts with \eqref{3.7}.

Let's rule out the non-vanishing case. By the discrete translation invariance, we can assume that $b_{k}=0$.
Since $||v^{k}||_{l^{2}_{k}}=1$, there exists $v=\{v_{l}\}$ such that $v^{k}_{l}\rightarrow v_{l}$ for all $l\in \mathbb{Z}$. It is obvious that $v\in l^{2}$, $||v||_{l^{2}}\leq 1$ and $|v_{0}|\geq \delta$.

Since $|v_{0}|\neq 0$, then $|u^{k}_{0}|\rightarrow\infty$, as $k\rightarrow\infty$. On the other hand, we have
$$
\frac{\sigma\beta}{\sigma+1}|u^{k}_{0}|^{2\sigma+2}\leq m_{k}\leq M'.
$$
It is a contradiction.
\end{proof}

\begin{theorem}
Assume that the frequency $\omega<0$ and $\alpha,\beta>0$.
There exists a positive localized ground state $u$ for the equation $\eqref{2.3}$.
\end{theorem}
\begin{proof}
Let $u^{k}\in l^{2}_{k}$ be a periodic ground state. From Lemma \ref{lem3.1}, the sequence $||u^{k}||_{l^{2}_{k}}$ is bounded.
Therefore, $u^{k}$ is either vanishing or non-vanishing. In the case of vanishing, we have $\lim_{k\rightarrow\infty}||u^{k}||_{l^{p}_{k}}\rightarrow 0$, for $p>2$. There holds that
\begin{align*}
|\omega|||u^{k}||^{2}_{l^{2}_{k}}&\leq \langle(-\triangle_{d}-\omega) u^{k},u^{k}\rangle_{k}\\
&= 2\alpha\sum_{l\in P_{k}}|u^{k}_{l}|^{2}|u^{k}_{l-1}|^{2}+\beta\sum_{l\in P_{k}}|u^{k}_{l}|^{2\sigma+2}\\
&\leq 2\alpha\sum_{l\in P_{k}}|u^{k}_{l}|^{4}++\beta\sum_{l\in P_{k}}|u^{k}_{l}|^{2\sigma+2}\rightarrow 0.
\end{align*}
It is a contradiction.

Thus, the sequence $u^{k}$ is non-vanishing. By the discrete translation invariance, we assume that $|u^{k}_{0}|\geq \delta>0$. There exists  $u=\{u_{l}\}$ such that $u^{k}_{l}\rightarrow u_{l}$ for all $l\in \mathbb{Z}$. It is obvious that $u\in l^{2}$ and $u\neq 0$. Also, we obtains that $u$ is a nontrival solution for \eqref{2.3} by point-wise limits. Now, we want to prove that $u$ is a localized ground state.

 Let $L$ be a positive integar such that
\begin{align*}
\liminf_{k\rightarrow\infty}J_{k}(u^{k})&\geq \liminf_{k\rightarrow\infty}\alpha\sum_{-L\leq l\leq L}|u^{k}_{l}|^{2}|u^{k}_{l-1}|^{2}+\frac{\sigma\beta}{\sigma+1}\sum_{-L\leq l\leq L}|u^{k}_{l}|^{2\sigma+2}\\
&\geq \alpha\sum_{-L\leq l\leq L}|u_{l}|^{2}|u_{l-1}|^{2}+\frac{\sigma\beta}{\sigma+1}\sum_{-L\leq l\leq L}|u_{l}|^{2\sigma+2}.
\end{align*}
Let $L\rightarrow\infty$, it  obtains that
$$
\liminf_{k\rightarrow\infty}J_{k}(u^{k})\geq J(u)\geq m,
$$
\begin{align}\label{3.8}
\liminf_{k\rightarrow \infty}m_{k}\geq m.
\end{align}
For any given $\epsilon>0$, let $u'\in N$ such that
$$
J(u')=\alpha\sum_{l\in\mathbb{Z}}|u'_{l}|^{2}|u'_{l-1}|^{2}+\frac{\sigma\beta}{\sigma+1}\sum_{l\in\mathbb{Z}}|u'_{l}|^{2\sigma+2}<m+\epsilon.
$$
Choose $t_{1}>1$ such that
$$
J(t_{1}u')<m+\epsilon,\quad I(t_{1}u')<0.
$$
From density argument, there exists a finite supported sequence $v=\{v_{l}\}$ sufficiently close to $t_{1}u'$ in $l^{2}$ such that

\begin{align*}
I(v)<0\quad \mbox{and}\quad \alpha\sum_{l\in\mathbb{Z}}|v_{l}|^{2}|v_{l-1}|^{2}+\frac{\sigma\beta}{\sigma+1}\sum_{l\in\mathbb{Z}}|v_{l}|^{2\sigma+2}<m+\epsilon.
\end{align*}
Thus, there exists  $t_{2}\in(0,1)$ such that $t_{2}v\in \mathcal{N}$ and $J(t_{2}v)< m+\epsilon$.

Choose $k$ large enough such that $P_{k}$ contains the support of $v$. Let $v^{k}\in l^{2}_{k}$ such that $v^{k}_{l}=t_{2}v_{l}$ for $l\in P_{k}$. It concludes that
$$
I_{k}(v^{k})=I(t_{2}v),
$$
$$
J_{k}(v^{k})=J(t_{2}v)<m+\epsilon.
$$
It implies
$$
\limsup_{k\rightarrow\infty} m_{k}<m+\epsilon.
$$

Combining with \eqref{3.8}, we have $\lim_{k\rightarrow\infty}m_{k}=m$. It completes the proof.
\end{proof}
\begin{remark}
With the similar argument in \eqref{2.6}, \eqref{2.7}, the power of the localized ground state has a lower bound $C_{1}>0$.  For more estimates, we refer to  \cite{Kara}.
\end{remark}

\section{Global convergence}
\begin{theorem}
Let $u^{k}\in l^{2}_{k}$ be the periodic ground state to the equation \eqref{2.3}. Then, there exists a ground state $u\in l^{2}$ such that $u^{k}$ strongly convergent to $u$ in $l^{2}_{k}$ after some discrete  translation.
\end{theorem}
\begin{proof}
Let $u^{k}\in l^{2}_{k}$ be the periodic ground state and $b_{k}\in\mathbb{Z}$. Now, we consider a translation
\begin{align*}
u'^{k}_{l}=u^{k}_{l+b_{k}}.
\end{align*}
From the argument above, we can assume that $u'^{k}_{l}\rightarrow u_{l}$ for all $l\in \mathbb{Z}$ where $u$ is a ground state. We want to prove that $||u'^{k}-u||_{l^{2}_{k}}$ convergent to $0$ as $k\rightarrow \infty$. First, it concludes that
\begin{align}
J_{k}(u'^{k}-u)\rightarrow 0,\quad I_{k}(u'^{k}-u)\rightarrow 0,\quad \mbox{as}\quad k\rightarrow \infty.
\end{align}
Indeed,
\begin{align*}
J_{k}(u'^{k}-u)=&\langle-\triangle_{d}(u'^{k}-u),(u'^{k}-u)\rangle-\omega\langle (u'^{k}-u),(u'^{k}-u)\rangle\\
&-\alpha\sum_{l\in P_{k}}|u'^{k}_{l}-u_{l}|^{2}|u'^{k}_{l+1}-u_{l+1}|^{2}-\frac{\beta}{\sigma+1}\sum_{l\in P_{k}}|u'^{k}_{l}-u_{l}|^{2\sigma+2}\\
=&J_{k}(u'^{k})-J_{k}(u)-2\langle-\triangle_{d}(u'^{k}-u),u\rangle-2\omega\langle (u'^{k}-u),u\rangle\\
&-\alpha\sum_{l\in P_{k}}|u'^{k}_{l}-u_{l}|^{2}|u'^{k}_{l+1}-u_{l+1}|^{2}-\frac{\beta}{\sigma+1}\sum_{l\in P_{k}}|u'^{k}_{l}-u_{l}|^{2\sigma+2}\\
&+\alpha\sum_{l\in P_{k}}|u'^{k}_{l}|^{2}|u'^{k}_{l+1}|^{2}+\frac{\beta}{\sigma+1}\sum_{l\in P_{k}}|u'^{k}_{l}|^{2\sigma+2}\\
&-\alpha\sum_{l\in P_{k}}|u_{l}|^{2}|u_{l+1}|^{2}-\frac{\beta}{\sigma+1}\sum_{l\in P_{k}}|u_{l}|^{2\sigma+2}.
\end{align*}

Similar with the argument in \cite{Pan1}, it obtains that $J_{k}(u'^{k})\rightarrow J(u)=m$ and $-2\langle-\triangle_{d}(u'^{k}-u),u\rangle-2\omega\langle (u'^{k}-u),u\rangle\rightarrow 0$, as $k\rightarrow\infty$.

Since that $||u'^{k}||_{l^{2}_{k}}$ and $||u||_{l^{2}}$ is bounded. For any given $\epsilon>0$, there exists $M>0$ such that
$\sum_{|l|\geq M}|u_{l}|^{2}<\epsilon$. Therefore, we have
\begin{align*}
&\alpha\!\!\!\!\!\sum_{l\in P_{k},|l|\geq M}\!\!\!\!\!|u'^{k}_{l}|^{2}|u'^{k}_{l+1}|^{2}-\alpha\!\!\!\!\!\sum_{l\in P_{k},|l|\geq M}\!\!\!\!\!|u'^{k}_{l}-u_{l}|^{2}|u'^{k}_{l+1}-u_{l+1}|^{2}\\
\leq &\alpha\!\!\!\!\!\sum_{l\in P_{k},|l|\geq M}\!\!\!\!\!(|u'^{k}_{l}|^{2}-|u'^{k}_{l}-u_{l}|^{2})|u_{l+1}|^{2}+\alpha\!\!\!\!\!\sum_{l\in P_{k},|l|\geq M}\!\!\!\!\!|u'^{k}_{l}-u_{l}|^{2}(|u'^{k}_{l+1}|^{2}-|u'^{k}_{l+1}-u_{l+1}|^{2})\\
\leq &\alpha\Big(2||u'^{k}||_{l^{2}_{k}}||u||_{l^{2}}+||u||^{2}_{l^{2}}\Big)\!\!\!\!\!\sum_{l\in P_{k},|l|\geq M}\!\!\!\!\!|u_{l+1}|^{2}+\alpha(||u'^{k}||^{2}_{l^{2}_{k}}+||u||^{2}_{l^{2}})(2||u'^{k}||_{l^{2}_{k}}+||u||_{l^{2}})\Big(\!\!\!\!\!\sum_{l\in P_{k},|l|\geq M}\!\!\!\!\!|u_{l+1}|^{2}\Big)^{\frac{1}{2}}\\
\lesssim&_{M',||u||^{2}_{l^{2}}} \epsilon
\end{align*}
for $k$ large enough.

Also, we have
\begin{align*}
-\frac{\beta}{\sigma+1}\sum_{l\in P_{k},|l|\geq M}|u'^{k}_{l}-u_{l}|^{2\sigma+2}+\frac{\beta}{\sigma+1}\sum_{l\in P_{k},|l|\geq M}|u'^{k}_{l}|^{2\sigma+2}\lesssim_{M',||u||^{2}_{l^{2}}} \epsilon
\end{align*}
for $k$ large enough.

On the other hand, from the point limits, it concludes that
\begin{align*}
&-\alpha\sum_{l\in P_{k},|l|<M}|u'^{k}_{l}-u_{l}|^{2}|u'^{k}_{l+1}-u_{l+1}|^{2}-\frac{\beta}{\sigma+1}\sum_{l\in P_{k},|l|<M}|u'^{k}_{l}-u_{l}|^{2\sigma+2}\\
&+\alpha\sum_{l\in P_{k},|l|<M}|u'^{k}_{l}|^{2}|u'^{k}_{l+1}|^{2}+\frac{\beta}{\sigma+1}\sum_{l\in P_{k},|l|<M}|u'^{k}_{l}|^{2\sigma+2}\\
&-\alpha\sum_{l\in P_{k},|l|<M}|u_{l}|^{2}|u_{l+1}|^{2}-\frac{\beta}{\sigma+1}\sum_{l\in P_{k},|l|<M}|u_{l}|^{2\sigma+2}<\epsilon
\end{align*}
for $k$ large enough.

Combine with H\"{o}lder inequality, there holds $J^{k}(u'^{k}-u)\rightarrow 0$ .

 With similar argument, we obtain $I^{k}(u'^{k}-u)\rightarrow 0$. Therefore,
\begin{align*}
&J^{k}(u'^{k}-u)-I^{k}(u'^{k}-u)\\
=&\alpha\sum_{l\in P_{k}}|u'^{k}_{l}-u_{l}|^{2}|u'^{k}_{l-1}-u_{l-1}|^{2}+\frac{\sigma\beta}{\sigma+1}\sum_{l\in P_{k}}|u'^{k}_{l}-u_{l}|^{2\sigma+2}\rightarrow 0.
\end{align*}
Since $||\cdot||_{l^{\infty}_{k}}\leq ||\cdot||_{l^{2\sigma+2}_{k}}$, we have $||u'^{k}-u||_{l^{\infty}_{k}}\rightarrow 0$. From Lemma \ref{lem3.1}, it is known that $||u'^{k}-u||_{l^{p}_{k}}\rightarrow 0$ for $p>2$. Hence,
\begin{align*}
&|\omega|||u'^{k}-u||^{2}_{l^{2}_{k}}\\
\leq &\langle-\triangle_{d}(u'^{k}-u),(u'^{k}-u)\rangle-\omega\langle (u'^{k}-u),(u'^{k}-u)\rangle\\
=&2\alpha\sum_{l\in P_{k}}|u'^{k}_{l}-u_{l}|^{2}|u'^{k}_{l+1}-u_{l+1}|^{2}+\beta\sum_{l\in P_{k}}|u'^{k}_{l}-u_{l}|^{2\sigma+2}\\
\leq&2\alpha||u'^{k}-u||^{4}_{l^{4}_{k}}+\beta||u'^{k}-u||^{2\sigma+2}_{l^{2\sigma+2}_{k}}\rightarrow 0.
\end{align*}
It completes the proof.

\end{proof}

\noindent\textit{\textbf{Acknowledgments.}}













\begin{thebibliography}{3}
\bibitem{Kara}
N. I. Karachalios, B. S\'{a}nchez-Rey, P. G. Kevrekidis and J. Cuevas,
Breathers for the Discrete Nonlinear Schr\"{o}dinger equation with nonlinear hopping

\bibitem{Pan1}
A. Pankov and V. Rothos, Periodic and decaying solutions in discrete
nonlinear Schr\"{o}dinger with
saturable nonlinearity, Proc. R. Soc. A (2008) 464, 3219-3236


\bibitem{Pan2}
Pankov, A. 2005a Travelling waves and periodic oscillations in Fermi-Pasta-Ulam lattices.
London, UK: Imperial College Press.

\bibitem{Pan3}
Pankov, A. 2005b Periodic nonlinear Schro¨dinger equation with an application to photonic
crystals. Milan J. Math. 73, 259-287.

\bibitem{Pan4}
Pankov, A. 2006 Gap solitons in periodic discrete NLS equations. Nonlinearity 19, 27-40.

\bibitem{Neh}
Nehari, Z. 1960 On a class of nonlinear second order differential equations. Trans. Am. Math. Soc.
95, 101-123.

\bibitem{Tes}
Teschl, G. 2000 Jacobi operators and completely integrable nonlinear lattices. Providence, RI:
American Mathematical Society.

\bibitem{Wein}
Weinstein, M. I., Excitation thresholds for nonlinear localized modes on lattices, Nonlinearity 12, 673 (1999).

\bibitem{Joh}
Johansson, M. and Aubry, S., Existence and stability of quasiperiodic breathers in the discrete nonlinear Schr\"{o}dinger equation, Nonlinearity 10, 1151 (1997).

\bibitem{Aub}
Aubry, S., Breathers in nonlinear lattices: existence, linear stability and quantization, Physica D 103, 201 (1997).

\bibitem{Zha}
Zhang, G. P., Breather solutions of the discrete nonlinear schr\"{o}dinger equations with unbounded potentials, J. Math.
Phys 50(1), 013505 (2009).

\bibitem{Lio}
Lions, P. L., The concentration compactness principle in the calculus of variations I: The locally compact case, Ann.
Inst. Henri Poincar\'{e}, Anal. Nonlin\'{e}aire 1, 223 (1984).

\bibitem{Cue}
Cuevas, J., Karachalios, N. I. and Palmero, F., Lower and upper estimates on the excitation threshold for breathers in
discrete nonlinear Schr¨odinger lattices, J. Math. Phys. 50(11), 112705 (2009)

\bibitem{Kara2}
Nikos I. Karachalios, Athanasios N. Yannacopoulos, Global existence and compact attractors for the
discrete nonlinear Schr\"{o}dinger equation, J. Differential Equations 217 (2005) 88-123


\end{thebibliography}
\end{document}